\documentclass[12pt]{amsart}

\usepackage{amscd,amsmath,latexsym,amsthm,amsfonts,amssymb,graphicx,geometry}

\usepackage{enumerate}

\allowdisplaybreaks

\usepackage[dvipsnames]{xcolor}
\definecolor{linkblue}{RGB}{0, 102, 204}     
\definecolor{softnavy}{RGB}{20, 60, 120}     

\usepackage{hyperref}
\hypersetup{
colorlinks=true,
linkcolor=linkblue,
citecolor=linkblue,
urlcolor=linkblue,   
}

\usepackage[norefs,nocites]{refcheck}

\usepackage{geometry}
\usepackage[normalem]{ulem}
\usepackage{cancel}

\usepackage{tikz-cd}

\newtheorem{theorem}{Theorem}[section]
\newtheorem{proposition}[theorem]{Proposition}

\newtheorem{corollary}[theorem]{Corollary}

\newtheorem{remark}[theorem]{Remark}

\newtheorem{problem}[theorem]{Problem}

\newtheorem{definition}[theorem]{Definition}
\numberwithin{equation}{section}
\newcommand{\N}{\mathbb{N}}
\newcommand{\R}{\mathbb{R}}
\newcommand{\C}{\mathbb{C}}
\newcommand{\K}{\mathbb{K}}

\begin{document}
\title[SOT-large subspaces of non-cyclic operators]{SOT-large subspaces of non-cyclic operators}

\author[Albuquerque]{N. G.~Albuquerque}
\address[N.~Albuquerque]{Departamento de Matem\'{a}tica \newline\indent
Universidade Federal da Para\'{i}ba \newline\indent
Jo\~ao Pessoa - PB \newline\indent
58.051-900 (Brazil)}
\email{\href{mailto:ngalbuquerque@mat.ufpb.br}{ngalbuquerque@mat.ufpb.br}}

\author[Ara\'{u}jo]{G.~Ara\'{u}jo}
\address[G.~Ara\'{u}jo]{Departamento de Matem\'{a}tica \newline\indent
Universidade Estadual da Para\'{i}ba \newline\indent
Campina Grande - PB \newline\indent
58.429-500 (Brazil)}
\email{\href{mailto:gustavoaraujo@servidor.uepb.edu.br}{gustavoaraujo@servidor.uepb.edu.br} }

\author[Bernal-Gonz\'alez]{L.~Bernal-Gonz\'alez}
\address[L.~Bernal-Gonz\'alez]{Departamento de An\'alisis Matem\'{a}tico \newline\indent
Instituto de Matem\'aticas Antonio de Castro Brzezicki (IMUS) \newline\indent
Universidad de Sevilla \newline\indent
Avda.~Reina Mercedes, 41012-Sevilla (Spain)}
\email{\href{mailto:lbernal@us.es}{lbernal@us.es}}

\author[Dem\'etrio Jr.]{E.~Dem\'etrio~Jr.}
\address[E.~Dem\'etrio~Jr.]{Departamento de Matem\'{a}tica \newline\indent
Universidade Federal da Para\'{i}ba \newline\indent
Jo\~ao Pessoa - PB \newline\indent
58.051-900 (Brazil)}
\email{\href{mailto:evandiodemetriojunior@gmail.com}{evandiodemetriojunior@gmail.com}}

\author[Seoane-Sep\'ulveda]{J.~Seoane-Sep\'ulveda}
\address[J.~Seoane-Sep\'ulveda]{Instituto de Matem\'atica Interdisciplinar (IMI) \newline\indent
Departamento de An\'alisis y Matem\'atica Aplicada \newline\indent
Facultad de Ciencias Matem\'aticas \newline\indent
Universidad Complutense de Madrid \newline\indent
Madrid \newline\indent
28040 (Spain)}
\email{\href{mailto:jseoane@mat.ucm.es}{jseoane@mat.ucm.es}}



\begin{abstract}
In this paper, we consider the strong operator topology (SOT) on the space $\mathcal{L} (X)$ of continuous linear operators on an infinite-dimensional Fr\'echet space $X$. The existence of SOT-dense subspaces as well as of SOT-closed infinite-dimensional subspaces inside the family of all non-cyclic operators is established. Other classes of operators, such as non-dense range or non-injective operators, are also studied in this regard. Moreover, we prove for the sequence space $\ell_p$ $(0 < p < \infty)$ that the family of all non-cyclic members of $\mathcal{L}(\ell_p)$ contains an isometric copy of $\ell_p$.
\end{abstract}

\maketitle


\section{Introduction and notation}

\quad In this paper, we focus on the algebraic-topological size --in a sense to be specified later-- of the class of all operators on a Fr\'echet space (in particular, on a Banach space) that do {\it not} enjoy cyclicity properties.

\vskip 3pt

Although our operators will mainly act in a Banach or Fr\'echet space, the concepts to be handled make sense in more general settings. Let $X$ be a
(Hausdorff) topological vector space
over the field $\K := \R$ or $\C$.
Let $\mathcal{L}(X)$ denote the vector space of all (linear, continuous) operators from $X$ into itself. We consider the following three dynamical properties of the members of $\mathcal{L}(X)$. If $x_0 \in X$ and $T \in \mathcal{L}(X)$, then the orbit and the projective orbit of $x_0$ under $T$ are defined respectively as the sets
\[
\mathcal{O}(T;x_0) := \{T^n x_0: n \ge 0\} \hbox{ \ \ and \ \ } \K \cdot \mathcal{O}(T;x_0) := \{\alpha T^n x_0: n \ge 0, \alpha \in \mathbb{K}\}.
\]
The operator $T$ is {\it hypercyclic} (resp., {\it supercyclic, cyclic}), when there exists a vector $x_0 \in X$ such that $\mathcal{O}(T;x_0)$  (resp., $\K \cdot (\mathcal{O}(T;x_0)$, ${\rm span}(\mathcal{O}(T;x_0))$) is dense in $X$, and $x_0$ is said a hypercyclic (resp., supercyclic, cyclic) vector for $T$.
It is plain that hypercyclicity implies supercyclicity and, in turn, supercyclicity implies cyclicity. Moreover, it is evident that if $X$ supports a cyclic (so, a supercyclic, or a hypercyclic) operator, then $X$ must be separable. While cyclicity is a long-standing notion in operator theory, hypercyclicity and supercyclicity are relatively new.

\vskip 3pt

For more information about hypercyclicity and supercyclicity we refer to the excellent books \cite{Bayart} and \cite{Grosse} (see also the survey \cite{grossesurvey}).

\vskip 3pt

For future references, we adopt the following notation:
\begin{itemize}
\item[$\bullet$] $\mathcal{CYC}(X) := \{T \in \mathcal{L}(X): \, T$ is cyclic$\}$
\item[$\bullet$] $\mathcal{SC}(X) := \{T \in \mathcal{L}(X): \, T$ is supercyclic$\}$
\item[$\bullet$] $\mathcal{HC}(X) := \{T \in \mathcal{L}(X): \, T$ is hypercyclic$\}$
\item[$\bullet$] $\mathcal{K}(X) := \{T \in \mathcal{L}(X): \, T$ is compact$\}$
\item[$\bullet$] $\mathcal{DR}(X) := \{T \in \mathcal{L}(X): \, T$ has dense range$\}$
\item[$\bullet$] $\mathcal{SRJ}(X) := \{T \in \mathcal{L}(X): \, T$ is surjective$\}$
\item[$\bullet$] $\mathcal{INJ}(X) := \{T \in \mathcal{L}(X): \, T$ is injective$\}$.
\end{itemize}

Recall that an operator $T \in \mathcal{L}(X)$ is said to be compact if the image of every bounded set is relatively compact in $X$.
Moreover, an F-space $X$ is a complete metrizable topological vector space, while $X$ is called a Fréchet space whenever it is
a locally convex F-space.

\vskip 3pt

The following properties can be found in the cited works \cite{Bayart}, \cite{grossesurvey} or \cite{Grosse}:
\begin{enumerate}
\item[$\bullet$] If $\mathcal{HC}(X) \ne \varnothing$ then $X$ must be infinite-dimensional.
\item[$\bullet$] If $X$ is a Fr\'echet (hence, if it is a Banach) separable infinite-dimensional space, then $\mathcal{HC}(X) \ne \varnothing$.
\item[$\bullet$] If $X$ is a Banach space (more generally, if $X$ is just a locally convex space: see \cite[Proposition 8]{bonetperis}) then
$\mathcal{K}(X) \cap \mathcal{HC}(X) = \varnothing$.
\end{enumerate}
In addition, the following inclusions are straightforward or trivial:
\begin{equation} \label{equation-inclusions}
\mathcal{HC}(X) \subset \mathcal{SC}(X) \subset \mathcal{DR}(X) \cap \mathcal{CYC}(X) \hbox{ \ and \ } \mathcal{SRJ}(X) \subset \mathcal{DR}(X).
\end{equation}
Furthermore, in the union
\[
\mathcal{DR}(X) \cup \mathcal{CYC}(X) \cup \mathcal{INJ}(X),
\]
none of the sets is redundant, that is, no class of operators is contained in the union of the remaining two in general. Below, we describe explicit examples of these situations, which will be essential for our main results in Section~\ref{Sec3} (see Theorems~\ref{Theorem-main-1} and~\ref{Theorem-main-2}).

First, we prove that there exist an infinite-dimensional Banach space \(X\) and a bounded linear operator \(T : X \to X\) such that \(T\) is surjective (hence has dense range), but is neither injective nor cyclic. Let $X = \ell^2 \oplus \mathbb{K}^2$ under the standard norm and $T(x,u) := (Bx, u)$, where $B: \ell^2 \to \ell^2$ is the backward shift operator $B(x_0, x_1, \ldots) = (x_1, x_2, \ldots)$ and $u \in \mathbb{K}^2$, then $T$ is surjective and satisfies $\ker(T) \neq \{0\}$ since $T(e_0,0) = (0,0)$ for $e_0 = (1,0,\ldots)$, hence $T$ is not injective, but since the orbits are given by $T^n(x,u) = (B^n x, u)$ for  $n \ge 0$. Every element in the span of the orbit of $(x,u)$ is of the form $p(T)(x,u) = (p(B)x, p(1)u)$ for some polynomial $p$, it follows that the second component of the orbit space always belongs to $\operatorname{span}\{u\}$, and thus $T$ is not cyclic.

Now we show that there exist injective operators that are neither dense-range nor cyclic. For instance, the operator  $(Tf)(x) = \sin(\pi x)\, f(x)$ on $C[0,1]$ is clearly injective, its range is not dense because $T(C[0,1]) \subset \{ f \in C[0,1] : f(0) = 0 = f(1) \}$. Now let us prove that for any given function $f \in C[0,1]$, ${\rm span}\left(\mathcal{O}(T;f)\right)$ is never dense. Suppose first that \(f\) is such that \(f(0) = 0\). Then the constant function \(g \equiv 1\) does not belong to the closure of ${\rm span}(\mathcal{O}(T;f))$. Indeed, if there existed a sequence
$h_n = \sum_{i=0}^{m(n)} b_i^{(n)} T^i(f) \in \operatorname{span} \left(\{T^j(f) : j \ge 0\}\right)$ ($n \in \mathbb{N}$)
such that \(h_n \to g\), we would obtain \(0 = h_n(0) \to g(0) = 1\), a contradiction. Now suppose that \(f(0) \neq 0\). If ${\rm span}(\mathcal{O}(T;f))$ were dense in \(C[0,1]\), then there would exist a sequence
$g_n = \sum_{i=0}^{m(n)} a_i^{(n)} T^i(f) \in \operatorname{span} \left(\{T^j(f) : j \ge 0\}\right)$
($n \in \mathbb{N}$) such that \(g_n \to I_d\), that is, \(g_n(x) \to x\) for every \(x \in [0,1]\). In particular, we would have
$g_n(0) = a_0^{(n)} f(0),\, g_n(1) = a_0^{(n)} f(1)$.
Hence,
$a_0^{(n)} f(0) \to 0 \quad \text{and} \quad a_0^{(n)} f(1) \to 1,$
which is impossible. Therefore, the operator \(T\) cannot be cyclic.

Finally, we exhibit a cyclic operator that fails to have dense range and to be injective. Consider the operator $T : C^\infty([0,1])\to C^\infty([0,1]), (Tf)(x) = x f'(x)$. Recall that a sequence $(f_n)_n$ converges to $f$ in $C^\infty([0,1])$ if and only if $f_n^{(k)} \to f^{(k)}$ uniformly on $[0,1]$, for every $k\ge0$. The operator \(T\) is not injective, since every constant function belongs to its kernel. Indeed, if \(f\equiv c\), then \(Tf(x)=x f'(x)=0\), hence \(\ker(T)\neq\{0\}\). Moreover, as previously, the range of \(T\) is not dense in \(C^\infty([0,1])\) because of $\operatorname{Im}(T) \subseteq \{f\in C^\infty([0,1]): \, f(0)=0\}$.
We now show that the function \(f(x)=e^x\) is a cyclic vector for \(T\). A direct computation gives
\(Tf(x)=xe^x\),
\(T^2f(x)=e^x(x^2+x)\),
and
\(T^3f(x)=e^x(x^3+3x^2+x)\).
More generally, $T^n f(x) = e^x p_n(x),$
where \(p_n\) is a polynomial of degree \(n\) satisfying \(p_n(0)=0\) and $T^0f(x)=e^x$. From the previous identities, we obtain
\(xe^x = Tf \in \operatorname{span}\mathcal{O}(T;f)\),
\(x^2e^x=T^2f-Tf \in \operatorname{span}\mathcal{O}(T;f)\),
and
\(x^3e^x = T^3f-3x^2e^x-xe^x \in \operatorname{span}\mathcal{O}(T;f)\).
Proceeding inductively, we conclude that
\[
x^n e^x \in \operatorname{span}\mathcal{O}(T;f)
\qquad \text{for every } n\in\mathbb N.
\]
It follows from the Weierstrass approximation theorem that
$f(x) = e^x$ is a cyclic vector for $T$. Consequently, $T$ is cyclic.

\vskip 3pt

We will also use some terminology extracted from the modern theory of lineability, initiated by V.I.~Gurariy at the beginning of the millennium.
The reader is referred to \cite{ABPS} for background on lineability. The aim of lineability is locating large algebraic structures inside
non necessarily linear sets. Assume that $X$ is a vector space and that $A \subset X$. Then $A$ is said to be
\begin{enumerate}
\item[$\bullet$] {\it lineable} provided that there exists an infinite-dimensional vector space $V$ such that $V \subset A \cup \{0\}$,
\item [$\bullet$] {\it pointwise lineable} if for each $x \in A$ there exists an infinite-dimensional vector space $V_x$ such that $x \in V_x \subset A \cup \{0\}$,
\item[$\bullet$] {\it algebrable} if $A$ is contained in some (linear) algebra and there exists an infinitely generated algebra $M$ such that $M \subset A \cup \{0\}$,
\item[$\bullet$] {\it dense-lineable} if $X$ is a topological vector space and there exists a dense vector subspace $V \subset X$ such that
                 $V \subset A \cup \{0\}$, and
\item[$\bullet$] {\it spaceable} if $X$ is a topological vector space and there exists a closed infinite-dimensional vector subspace $V \subset X$ such that
                 $V \subset A \cup \{0\}$.
\end{enumerate}
Of course, each one of the properties algebrability, spaceability, or dense-lineability (in this case, if $X$ is infinite-dimensional), implies lineability.

\vskip 3pt

A natural topology on $\mathcal{L}(X)$ is the so-called {\it strong operator topology} (SOT) (see, e.g., \cite[Chapter 2]{simon}). In it, a net $(T_\alpha)_{\alpha \in I}$
tends to $T$ if and only if $T_\alpha (x) \to T(x)$ for every $x \in X$. In the case that $X$ is a Banach space, it is a well-known fact that the SOT topology
is weaker than the norm topology on $\mathcal{L}(X)$ (the norm is $\|T\| = \sup_{\|x\|=1} \|T(x)\|$), and that both topologies coincide if and only if $X$ is finite dimensional.

\vskip 3pt

This paper is organized as follows. In section 2 we will make some con\-si\-de\-ra\-tions about the linear or topological size of the class of hypercyclic ope\-ra\-tors on a Banach space $X$ under the norm topology or the SOT-topology on $\mathcal{L}(X)$. This will motivate the analysis of the size of the {\it complement} \,of this class (and of other related ones) in the space $\mathcal{L}(X)$. In section 3 we establish our main results, namely, the family of all operators on an infinite-dimensional 
Fr\'echet space $X$ that are neither cyclic nor dense-range nor injective, contains, except for zero, a SOT-dense subspace as well as a SOT-closed infinite-dimensional subspace of $\mathcal{L}(X)$. In section 4 we prove that, for every $p \in (0,\infty )$, the family of all non-cyclic or non-dense range members of \,$\mathcal{L}(\ell_p)$ \,contains an isometric copy of the sequence space $\ell_p = \left\{(x_n) \in \K^{\N} : \, \sum_{n=1}^{\infty} |x_n|^p <\infty \right\}$.
Finally, and in connection with the results in section 3, we have incorporated in section 5 a statement about the lineability of the class of operators that are
not $N$-supercyclic.

\section{Size of the sets of hypercyclic and non-hypercyclic operators}

\quad In 1969 Rolewicz \cite{rolewicz} provided the first example of a hypercyclic ope\-ra\-tor on a Banach space.
Specifically, he proved that if
\[
B : (x_n) \in \ell_p \longmapsto (x_{n+1}) \in \ell_p
\]
is the backward shift on the sequence space $\ell_p$
under the norm \,$\| (x_n) \|_p = (\sum_{n=1}^{\infty} |x_n|^p)^{1/p}$, where \,$1 \le p < \infty$, then any scalar multiple \,$\lambda B$
\,is hypercyclic on $\ell_p$ \,if \,$|\lambda | > 1$. Despite \,$\ell_p$ \,is not locally convex (hence not Fr\'echet)
for \,$0 < p < 1$, by using the so-called Hypercyclicity Criterion (see, e.g., \cite[Theorem 1.6]{Bayart}) we can obtain that
these Rolewicz operators are also hypercyclic on \,$\ell_p$ \,when endowed with the F-norm $\| (x_n) \|_p = \sum_{n=1}^{\infty} |x_n|^p$.

\vskip 3pt

Since the publication of Rolewicz's paper, many advances on hy\-per\-cy\-cli\-ci\-ty have taken place.
Once a great deal of concrete examples of hypercyclic operators has been given, a natural arising question is the size
of the fa\-mi\-ly of these special operators within the space of all operators.

\vskip 3pt

Concerning topological size, B\`es and Chan \cite{beschan,beschan2,Chan} (see also \cite{prajitura}) were able to prove that the set
of hypercyclic operators on an infinite-dimensional separable Fréchet space $X$ is SOT-dense in $\mathcal{L}(X)$.
On the contrary, Wu \cite{Wu} had established in 1994 that the class of all cyclic operators (hence, the class of hypercyclic operators) on a
Banach space is nowhere dense under the operator norm topology. Therefore, the family of non-cyclic operators (hence, the family of non-hypercyclic operators) is ``huge''
(residual and, in particular, dense) in this topology. In fact, the denseness of this family in the Hilbert space case is known since 1972
\cite{fillmoresw}. A spectral description of the norm-closure of $\mathcal{HC}(X)$ on an infinite-dimensional separable complex Hilbert space $X$ was furnished by Herrero \cite{Herrero} in 1991.

\vskip 3pt

Regarding algebraic size, it is plain that the set $\mathcal{HC}(X)$ is never lineable if $X$ is a normed space. Indeed, every $T \in \mathcal{HC}(X)$ must satisfy $\|T\| > 1$ (otherwise, all orbits under $T$ would be bounded, so non-dense). Then $(1 + \|T\|)^{-1} T \not\in \mathcal{HC}(X)$ and, consequently, $\mathcal{HC}(X) \cup \{0\}$ does not even contain a one-dimensional subspace.

\vskip 3pt

Nevertheless, a non-normable topological vector space may well admit large vector spaces consisting, except for zero, of hypercyclic operators. For instance, in 1991 Godefroy and Shapiro \cite{godefroyshapiro} established that, for every nonconstant entire function $\Phi$ of exponential type, the associated differential operator $\Phi (D)$ is hypercyclic on the space $H(\C )$ of entire functions on the complex plane, that is a Fr\'echet space when endowed with the compact-open topology. Then $\mathcal{HC}(H(\C ))$ is $\mathfrak{c}$-lineable, that is, it contains, except for zero, a $\mathfrak{c}$-dimensional vector space, where $\mathfrak{c}$ denotes the cardinality of the continuum.

\vskip 3pt

In any case, it is natural to investigate the lineability of the complementary family $\mathcal{L}(X) \setminus \mathcal{HC}(X)$.
The answer is positive for locally convex spaces and, in a stronger way, for Banach spaces:
%
%
%
%

\begin{proposition} \label{proposition}
\begin{enumerate}
\item[\rm (a)] For every infinite-dimensional locally convex space \(X\), the set \(\mathcal{L}(X) \setminus \mathcal{HC}(X)\) is lineable.
\item[\rm (b)] For every infinite-dimensional Banach space \(X\), the set 
\(\mathcal{L}(X) \setminus \mathcal{HC}(X)\) is spaceable with respect to the operator norm topology.
\end{enumerate}
\end{proposition}

\begin{proof}
For (a) and (b), we have by a result mentioned in section 1 (\cite[Proposition 8]{bonetperis}) that
\[
\mathcal{K}(X) \subset \mathcal{L}(X) \setminus \mathcal{HC}(X).
\]
Recall that $\mathcal{K}(X)$ is a vector subspace of $\mathcal{L}(X)$. Since it is closed in $\mathcal{L}(X)$ for the operator norm topology if $X$ is Banach, it suffices to show that $\mathcal{K}(X)$ contains infinitely many linearly independent members whenever $X$ is locally convex. With this aim, take a linearly independent infinite set $\{x_1,x_2,x_3, \dots \} \subset X$.
Since each linear ${\rm span} \left(\{x_n\}\right)$ ($n \in \N$) is finite dimensional, we have (see, e.g., \cite[p.~106]{RudinAF}) that there exists a closed subspace $M_n$ of $X$ such that $X = {\rm span} \{x_n\} \oplus M_n$ (topological direct sum). For each $x \in X$, we get a decomposition $x = \lambda x_n + y_n$ with unique $\lambda \in \K$ and $y_n \in M_n$, and the projection $x \in X \mapsto \lambda \in \K$ is continuous (that is, it belongs to $X^*$). Then the mappings $K_n : x \in X \mapsto \lambda x_n \in X$
($n \in \N$) are well defined, linear, continuous, linearly independent (because the $x_n$'s are) and compact (because their ranges are finite dimensional). The proof is concluded.
\end{proof}


The next question arises naturally.

\begin{problem}
For every infinite-dimensional Banach space \(X\), is the set \(\mathcal{L}(X) \setminus \mathcal{HC}(X)\) dense-lineable (with respect to the operator norm topology)?
\end{problem}

The second part of the last proposition will be improved in section 3 (see Corollary \ref{corollary}). In the emblematic case $X = \ell_p$, the result can be reinforced with the obtaining of isometric images, as it will be seen in section 4.
Notice that the approach of the preceding proof cannot be applied to $\mathcal{L}(X) \setminus \mathcal{SC}(X)$ (hence, neither to $\mathcal{L}(X) \setminus \mathcal{CYC}(X)$), because a supercyclic operator may well be compact on a Banach space, as for instance the operator provided by Herzog in \cite{Herzog}.

\vskip 3pt

The following result provides that pointwise lineability never occurs for the set \(\mathcal{L}(X) \setminus \mathcal{HC}(X)\), in the context of infinite dimensional separable Banach spaces.
We recall that a subset $C$ of a vector space is a cone if $sx \in C$ for every $s > 0$ and $x \in C$.

\begin{proposition} \label{no-pointwise}
For every infinite-dimensional separable Banach space \(X\), the set \(\mathcal{L}(X) \setminus \mathcal{HC}(X)\) fails to contain a cone containing some vector of this set, and hence it cannot be pointwise lineable.
\end{proposition}

\begin{proof}
Since $X$ is separable, we can select an operator $S \in \mathcal{HC}(X)$. Then $T := {S \over 1 + \|S\|}$ is a contraction, hence
$T \in \mathcal{L}(X) \setminus \mathcal{HC}(X) =: \mathcal{A}$. Let $\alpha := {1 \over 1 + \|S\|} > 0$.
Thus, $S = \alpha T \not\in \mathcal{A}$, which proves the assertion.
\end{proof}

Now, with the target set on algebraic-topological properties of $\mathcal{L}(X) \setminus \mathcal{HC}(X)$ with respect to the SOT-topology, we also can take advantage of the class of compact operators. In fact, by using the subclass of finite rank operators, the first part of Proposition \ref{proposition} will be refined in section 3 so as to obtain dense lineability (see Theorem \ref{Theorem-main-1}).

\begin{remark}
{\rm Before facing more general situations, let us show how in the case of an infinite-dimensional real Hilbert space $H$, self-adjoint operators provide
an easy proof of the SOT-spaceability of
\(\mathcal{L}(H) \setminus \mathcal{HC}(H)\).
For each operator $T \in \mathcal{L}(H)$, denote by \,$T^*$ its adjoint operator.
Consider the class
\[
SA(H) := \{ T \in \mathcal{L}(H) : T=T^{*}\}
\]
of self-adjoint operators.
Plainly, $SA(H)$ is a vector subspace of $\mathcal{L}(H)$. Moreover, it has infinite dimension
as the operators $T_n(x) := e_n^*(x)e_n$ $(n \in \N )$ (where $(e_n)$ is an orthonormal system
and $(e_n^*)$ is the sequence of its corresponding coordinate functionals) are linearly independent members of $SA(H)$.
Since no self-adjoint operator can be hypercyclic \cite[Corollary 5.31]{Grosse}, we get
\,$SA(H) \subset \mathcal{L}(H) \setminus \mathcal{HC}(H)$.
Then it is enough to show that $SA(H)$ is SOT-closed in $\mathcal{L}(H)$.
For this, let $\left(T_\alpha \right)_{\alpha \in I}$ be a net in $SA(H)$ such that $T_\alpha \stackrel{SOT}\longrightarrow T \in \mathcal{L}(X)$.
We have to prove that $T$ is self-adjoint. We have:
\[
\langle Tx,y \rangle
= \langle \lim_{\alpha} T_\alpha x,y \rangle
= \lim_{\alpha} \langle T_\alpha x , y \rangle
= \lim_{\alpha} \langle x, T_\alpha y \rangle
= \langle x, \lim_{\alpha}T_\alpha(y)
= \langle x,Ty \rangle,
\]
which proves our claim.}
\end{remark}

\section{Dense lineability and spaceability of operators not enjoying cyclicity or range properties}  \label{Sec3}

\quad In this main section, it will be seen how under rather natural assumptions on the space $X$ it is possible to obtain SOT-large spaces of operators not enjoying several dynamical or range properties.

\vskip 3pt

In order to motivate the introduction of properties of univalence and dense range, let us recall some recent results about the linear size of ope\-ra\-tors satisfying any of these properties:
\begin{enumerate}
\item[$\bullet$] Let $X$ be a separable infinite-dimensional Banach space. Then $\mathcal{INJ}(X)$ is algebrable \cite[Theorem 3.7]{bernalopen}.

\item[$\bullet$] If $X$ is a topological vector space admitting a Schauder basis, then $\mathcal{INJ}(X)$ is lineable \cite[Theorem 3.1]{aronbjms}.

\item[$\bullet$] If $c_0$ denotes the Banach space of scalar sequences tending to $0$, endowed with the supremum norm, then $\mathcal{INJ}(c_0)$ is spaceable in $\mathcal{L}(c_0)$ under the operator norm topology \cite[Corollary 3.4]{aronbjms}. More generally, if $Y \ne \{0\}$ is a Banach space and $X \subset Y^{\N}$ is a standard Banach sequence space, then $\mathcal{INJ}(X)$ is norm-spaceable in $\mathcal{L}(X)$ \cite[Theorem 1.3]{diniz}.

%
%
%

\item[$\bullet$] If $p \in [1,+\infty ]$ then the set $\mathcal{SRJ}(\ell_p)$ (hence $\mathcal{DR}(\ell_p)$) is spaceable in $\mathcal{L}(\ell_p)$ under the operator norm topology \cite[Theorem 4.1]{aronbjms}. More generally, if $Y \ne \{0\}$ is a Banach space and $X \subset Y^{\N}$ is a standard Banach sequence space such that the set $c_{00}(Y)$ of finite sequences is dense in $X$, then $\mathcal{SRJ}(X)$ (hence $\mathcal{DR}(X)$) is norm-spaceable in $\mathcal{L}(X)$ \cite[Theorem 1.2]{diniz}. See also an extension of this result to a class of Banach function spaces called densely $(\Omega ,Y)$-spaces in
\cite[Theorem 2.5]{bagheri}.
\end{enumerate}
Then it is interesting to study the linear size of the {\it complements} of these classes.
In fact, we will be able to synchronize the lack of all three properties: cyclicity, dense range, and injectivity.

\begin{theorem} \label{Theorem-main-1}
Let \,$X$ be an infinite-dimensional 
locally convex space. Then the set
$$\mathcal{L}(X) \setminus (\mathcal{DR}(X) \cup \mathcal{CYC}(X) \cup \mathcal{INJ}(X))$$
is SOT-dense-lineable in $\mathcal{L}(X)$.
\end{theorem}

\begin{proof}
Consider the family $\mathcal{FR}(X)$ of finite rank operators, that is, an operator $T \in \mathcal{L}(X)$ belongs to $\mathcal{FR}(X)$
if and only if ${\rm dim} (T(X)) < \infty$. Since $\mathcal{FR}(X)$ is itself a vector space, the proof would be finished as soon as we prove the following four properties:
\begin{enumerate}
\item[$\bullet$] Every $T \in \mathcal{FR}(X)$ has not dense range.
\item[$\bullet$] Every $T \in \mathcal{FR}(X)$ is not injective.
\item[$\bullet$] Every $T \in \mathcal{FR}(X)$ is not cyclic.
\item[$\bullet$] $\mathcal{FR}(X)$ is SOT-dense in $\mathcal{L}(X)$.
\end{enumerate}

Assume that $T \in \mathcal{FR}(X)$. Since ${\rm dim} (T(X)) < \infty$, the image $T(X)$ is closed in $X$, and so $\overline{T(X)} = T(X) \ne X$ because $X$ is infinite-dimensional but $T(X)$ is not. Thus, $T$ has not dense range. If $T$ were injective then ${\rm dim}(T(X)) = {\rm dim}(X)$, which again is false. If $T$ were cyclic,
there would exist a vector $x_0 \in X$ such that $L := {\rm span}(\mathcal{O}(T;x_0))$ is dense in $X$. But $L \subset {\rm span} (\{x_0\} \cup T(X))$.
Therefore, ${\rm dim}(L) \le 1 + {\rm dim}(T(X)) < \infty$, and so $\overline{L} = L \ne X$ (because, again, $X$ is infinite-dimensional), that is a contradiction.

\vskip 3pt

Thus, our unique task is to prove the SOT-denseness of $\mathcal{FR}(X)$ in $\mathcal{L}(X)$. With this aim, fix a $T_0 \in \mathcal{L}(X)$ as well as an $\varepsilon > 0$ and
finitely many vectors $x_1, x_2, \dots, x_N \in X$. Consider the corresponding SOT-neighbourhood $U$ of $T_0$, given by
$$U = \{S \in \mathcal{L}(X): \, |S(x_j) - T_0(x_j)| < \varepsilon \hbox{ \ for all \ } j = \{1,2, \dots, N\} \}.$$
Without loss of generality, we can assume that some $x_j \ne 0$.
Select a ma\-xi\-mal linearly free system $\{y_1,y_2, \dots ,y_p\} \subset \{  x_1, x_2, \dots, x_N \}$. Then every $x_j$ is a finite linear combination of the vectors
$y_k$'s. Note that if $S \in \mathcal{L}(X)$ satisfies $y_k \in {\rm ker} (S - T_0)$ for all $k \in \{1, \dots ,p\}$, then $S \in U$.
Consequently, it suffices to exhibit a finite rank operator \,$S$ \,satisfying this condition. Given $k \in \{1, \dots, p\}$, with the help of
the Hahn-Banach theorem we can obtain a functional $\varphi_k \in X^*$ such that $\varphi_k (y_k) = 1$ and $\varphi_k(y_i) = 0$ ($i \ne k$).
Then the operator
$$
S := \sum_{i=1}^{p} \varphi_i (\cdot ) \cdot T_0(y_i) \in \mathcal{L}(X)
$$
satisfies for each $k \in \{1, \dots, N\}$
that $(S - T_0)(y_k) = T_0(y_k) - T_0(y_k) = 0$. This concludes the proof.
\end{proof}

By considering an appropriate countable subfamily of $\mathcal{FR}(X)$, we will get spaceability in the Fréchet space case. Prior to establish our spaceability result (Theorem \ref{Theorem-main-2}), we recall the following crucial result on existence of basic sequences, that is due to S.~Mazur in the Banach case (see \cite[p.~39]{Diestel}), and to Bessaga, Kadec and Pe\l czy\'nski in the general Fréchet case (see \cite[pp.~140--143]{kamtangupta}). Recall that a sequence $(u_n)$ in a topological vector space $X$ is called a basic sequence if it is a Schauder basis for $Y := \overline{\rm span} (\{u_n : \, n \in \N \})$, that is, if for every $u \in Y$ there exists a unique sequence $(a_n) \in \K^{\N}$ such that $u = \sum_{n=1}^{\infty} a_n u_n = \lim_{n \to \infty} \sum_{k=1}^{n} a_k u_k$ and, in addition, the coordinate functionals
$$
e_n^* : u \in Y \longmapsto a_n \in \K \quad (n \in \N)
$$
are continuous (i.e.~$e_n \in Y^*$).

\begin{theorem} \label{MazurBessagaKadecPelczynski}
Every infinite-dimensional Fr\'echet space contains a basic sequence.
\end{theorem}

\begin{theorem} \label{Theorem-main-2}
Let \,$X$ be an infinite-dimensional 
Fr\'echet space. Then the set
$$\mathcal{L}(X) \setminus (\mathcal{DR}(X) \cup \mathcal{CYC}(X) \cup \mathcal{INJ}(X))$$
is SOT-spaceable in $\mathcal{L}(X)$.
\end{theorem}

\begin{proof}
According to Theorem \ref{MazurBessagaKadecPelczynski}, there exists a basic sequence \,$(x_n)_{n=1}^\infty$ \,in \,$X$. For each $n \in \mathbb{N}$, let us
denote by $e^*_n$ the coordinate functionals corresponding to this basic sequence, continuously extended to $X$ via the Hahn-Banach theorem.
Define the mappings
$$
K_n : x \in X \longmapsto e^*_n(x) \cdot x_n \in X \quad (n \in \N).
$$

It is plain, on the one hand, that every $K_n$ is linear and continuous, that is, $K_n \in \mathcal{L}(X)$. In fact, they are finite rank operators.
On the other hand, it is evident that
$$
\mathcal{M} :=  \overline{ {\rm span} (\{K_n: n \ge 3 \} ) }^{\rm SOT}
$$
is a SOT-closed vector subspace of $\mathcal{L}(X)$. Moreover, from the shape of $K_n$'s and the fact that any basic sequence is linearly independent,
one derives that $\mathcal{M}$ is infinite-dimensional.

\vskip 3pt

Therefore, we would be done as soon as we prove the following properties:
\begin{enumerate}
\item[$\bullet$] Every $T \in \mathcal{M}$ is not injective.
\item[$\bullet$] Every $T \in \mathcal{M}$ has not dense range.
\item[$\bullet$] Every $T \in \mathcal{M}$ is not cyclic.
\end{enumerate}

\vskip 3pt

Let \,$T \in \mathcal{M}$. Then there is a net \,$(T_\alpha )_{\alpha \in I} \subset {\rm span} (\{K_n: n \ge 3 \} )$ \,such that \,$T_\alpha \xrightarrow{\alpha \in I} T$
\,in the SOT-topology, that is, $T_\alpha (x) \xrightarrow{\alpha \in I} T(x)$ \,for every \,$x \in X$. Now, $x_1 \ne 0$
\,but \,$K_n(x_1) = e^*_n(x_1) \cdot x_n = 0 \cdot x_n = 0$ \,for all $n \ge 3$, and so \,$T_\alpha (x_1) = 0$ \,for all \,$\alpha \in I$.
It follows that \,$T(x_1) = \lim_{\alpha \in I} T_\alpha (x_1) = 0$. Thus, $T$ is not injective.

\vskip 3pt

Assume, by way of contradiction, that \,$x_1 \in \overline{T(X)}$. Then there would exist a sequence \,$(y_j) \subset X$ \,such that \,$T(y_j) \to x_1$
\,as \,$j \to \infty$. Observe that \,$e_1^*(K_n(x)) = e_n^*(x) \cdot e_1^*(x_n) = 0$ \,for all \,$x \in X$ \,and all \,$n \ge 3$. Therefore, by linearity, we obtain
\,$e_1^* (T_\alpha (x)) = 0$ \,for all \,$\alpha \in I$ \,and all \,$x \in X$. Now, the continuity of \,$e_1^*$ \,yields
$$
e_1^*(T(y_j)) = \lim_{\alpha \in I} e_1^* (T_\alpha (y_j)) = \lim_{\alpha \in I} 0 = 0 \hbox{ \ for all \ } j \in \N .
$$
It follows, again by the continuity of $e_1^*$, that \,$0 = e_1^*(T(y_j)) \to e_1^*(x_1) = 1$ \,as \,$j \to \infty$, which is absurd. Thus, $T$ has not dense range.

\vskip 3pt

Finally, let us suppose that \,$T$ \,is cyclic. As in the preceding paragraph, we can obtain that \,$e_1^* \circ T_\alpha = 0 = e_2^* \circ T_\alpha$ \,for all \,$\alpha \in I$.
Then the continuity of \,$e_1^*, e_2^*$ \,together with the fact \,$T_\alpha (x) \xrightarrow{\alpha \in I} T(x)$ \,for all \,$x \in X$
\,leads us to \,$e_1^* \circ T = 0 = e_2^* \circ T$ \,and, consequently,
\begin{equation}\label{Eq composition with T^n}
 e_1^* \circ T^n = 0 = e_2^* \circ T^n \hbox{ \ for all \ } n \in \N .
\end{equation}

The assumption of cyclicity implies the existence of
a vector $u \in X$ such that,
for every prescribed vector $v \in X$, there exist a sequence \,$(w_j)_{j=1}^\infty \subset {\rm span}(\mathcal{O}(T;u))$ \,satisfying
\,$w_j \to v$ \,as \,$j \to \infty$.
Every $w_j$ has the form
$$
w_j = \mu_j u + \sum_{n=1}^{m(j)} c_{n,j} T^n u
$$
for certain scalars \,$\mu_j$ and $c_{n,j}$ ($n \in \{1, \dots, m(j)\}, \, j \in \N$).
Let \,$z_j := \sum_{n=1}^{m(j)} c_{n,j} T^n u$. It follows from \eqref{Eq composition with T^n} by linearity that \,$e_1^*(z_j) = 0 = e_2^*(z_j)$, and so
$e_i^*(w_j) = \mu_j e_i^*(u)$ ($i=1,2$) \,for all \,$j \in \N$.

\vskip 3pt

Now, the continuity of \,$e_1^*$ \,and \,$e_2^*$ \,entails
$$
\mu_j e_i^*(u) = e_i^* (w_j) \longrightarrow e_i^*(v) \ \,(i=1,2) \hbox{ \ as \ } j \to \infty.
$$
If $e_1^*(u) = 0$, by taking $v = x_1$ we would have \,$0 \to e_1^*(x_1) = 1$ \,as \,$j \to \infty$, which is absurd.
If $e_1^*(u) \ne 0$, the choice $v = x_2$ would yield \,$\mu_j e_1^*(u) \to e_1^*(x_2) = 0$, so \,$\mu_j \to 0$.
But also \,$\mu_j e_2^*(u) \to e_2^*(x_2) = 1$, which contradicts \,$\mu_j \to 0$.
In any case, we reach a contradiction, and so \,$T$ \,cannot be cyclic. The proof is finished.
\end{proof}

The following consequence of Theorem \ref{Theorem-main-1} is extracted from the inclusions \eqref{equation-inclusions} and from the fact that, if $X$ is a Banach space, then the SOT topology is weaker than the operator norm topology.

\begin{corollary} \label{corollary}
Let \,$X$ be an infinite-dimensional 
Fr\'echet space. Then the sets \break
$\mathcal{L}(X) \setminus \mathcal{DR}(X), \, \mathcal{L}(X) \setminus \mathcal{CYC}(X), \,  \mathcal{L}(X) \setminus \mathcal{INJ}(X)$,
\, $\mathcal{L}(X) \setminus \mathcal{SRJ}(X), \, \mathcal{L}(X) \setminus \mathcal{SC}(X)$ \,and \,$\mathcal{L}(X) \setminus \mathcal{HC}(X)$
\,are SOT-dense-lineable and SOT-spaceable in $\mathcal{L}(X)$. In particular, if \,$X$ is an infinite-dimensional Banach space, then all of these sets are spaceable in
$\mathcal{L}(X)$ with respect to the operator norm topology.
\end{corollary}

\begin{remark}
{\rm Regarding {\it non-injectivity,} the following valuable related result has been obtained by Aires and Botelho in \cite[Example 2.2 and Corollary 4.1]{botelho-aires}:
Let $0 < p \le 1 \le q <\infty$, let $E$ be a $p$-Banach space such that $\dim E > 1$ and $E^* \neq \{0\}$, and let $F$ be a Banach space. Also, denote by $X$ one of the sequence classes $\ell_q,\, c_0$ or $\ell_\infty$, and by $X(F)$ the corresponding sequence space with values in $F$. Then the set of non-injective bounded linear operators from $E$ to $X(F)$ is pointwise spaceable in $\mathcal{L} \left( E, X(F) \right)$. Other findings about pointwise spaceability of families of (not ne\-ce\-ssa\-ri\-ly linear) non-injective mappings can be seen in \cite[Proposition 4.7 and Corollary 4.9]{botelho-aires}.
}
%
%
\end{remark}

\section{Spaceability of the family of non-cyclic operators on sequence spaces}

\quad It follows from Corollary \ref{corollary} that for an infinite-dimensional Banach space $X$ the family of non-cyclic operators is norm-spaceable in $\mathcal{L}(X)$.
In the case that $X$ is the sequence space $\ell_p$ ($1 \le p < \infty$), we can improve this result in the way
given in the next theorem.
In fact, we can also get that the operators under analysis do not have dense range.
In addition, the result will also hold in the non-locally convex case $0 < p < 1$.
In this case, we consider in $\mathcal{L}(\ell_p)$ the distance $d(T,S) = \|T - S\|$ generated by the mapping
$\|T\| := \sup \{\|T(x)\|_p : \, \|x\|_p \le 1\}$, where $\|(x_n)\|_p = \sum_{n=1}^{\infty} |x_n|^p$.
It is easy to check that $\| \cdot \|$ is an F-norm generating the topology of uniform convergence on bounded subsets of $\ell_p$.


\begin{theorem}\label{theo4.1}
For every  \,$p > 0$, the set \,$\mathcal{INJ}(\ell_p) \setminus (\mathcal{CYC} (\ell_p) \cup \mathcal{DR}(\ell_p))$ contains, except for zero, an isometric copy of \,$\ell_p$.

\end{theorem}

\begin{proof}
Let us take a partition of $\widetilde{\mathbb{N}} = \{3,4,5, \dots \}$ into countably many infinite, pairwise disjoint, subsets $\mathbb{N}_k$ ($k = 1,2, \dots$).
For each integer $k \in \widetilde{\mathbb{N}}$, we write $\mathbb{N}_k = \{n_1^{(k)} < n_2^{(k)} < \cdots\}$ and define $F_k: \ell_p \to \ell_p$ by
\[
F_k((x_i)_{i=1}^\infty) = (y_j)_{j=1}^\infty,
\]
where
\[
y_j = \sum_{i \in \mathbb{N}} x_i \, \delta_{j, n_i^{(k)}}.
\]
Here, as usual, $\delta_{m,n} = \left\{\begin{array}{ll}
                     1 \, \hbox{ \ if } \, m = n \\
                    0 \, \hbox{ \ if } \, m \ne n .
                     \end{array}\right.$

In other words, $F_k$ maps the entire sequence $x =(x_i)$ into the coordinates indexed by $\mathbb{N}_k$, preserving the original values of $(x_i)$ but relocating them to the positions in $\mathbb{N}_k$. Clearly, $F_k \in \mathcal{L}(\ell_p)$  for each $k \in \mathbb{N}$.
In fact, each $F_k$ is an isometry on $\ell_p$, that is, $\|F_k(x)\|_p = \|x\|_p$ for all $x = (x_i)_{i=1}^\infty \in \ell_p$.

\vskip 3pt

Now we define the mapping \,$\Phi : \ell_p \longrightarrow \mathcal{L}(\ell_p)$ \,by
\[
\Phi (\lambda ) := \sum_{k=1}^\infty \lambda_k F_k \quad \text{for each } \,\, \lambda = (\lambda_k)_{k=1}^\infty \in \ell_p.
\]
Since, for all $x \in \ell_p$, the supports of the sequences $F_k(x)$ and $F_l(x)$ are disjoint if $k \neq l$, the mapping \,$\Phi$ \,is well defined.
Clearly, $\Phi$ is linear and, if we denote
\,$\displaystyle{m(p) = \left\{\begin{array}{ll} p & \mbox{if
$p \ge 1$} \\ 1 & \mbox{if $0 < p < 1$,} \end{array} \right.}$
then we have
\begin{equation*}
\begin{split}
\|\Phi (\lambda )(x)\|_p^{m(p)} &= \left\| \sum_{k=1}^\infty \lambda_k F_k(x) \right\|_p^{m(p)} = \sum_{k=1}^\infty |\lambda_k|^p \|F_k(x)\|_p^{m(p)} \\
                                      &=\sum_{k=1}^\infty |\lambda_k|^p \|x\|_p^{m(p)} = (\|\lambda\|_p \|x\|_p)^{m(p)}.
\end{split}
\end{equation*}
Consequently, $\|\Phi (\lambda )(x)\|_p = \|\lambda\|_p \|x\|_p$ \,for all \,$x \in \ell_p$.
Hence  $\|\Phi (\lambda )\| = \|\lambda\|_p$ for all $\lambda \in \ell_p$, which tells us
that \,$\Phi$ \,is an isometry. Therefore, we also have that $\Phi(\lambda)$ is injective except for $\lambda = (0,0,0,\dots)$, and thus $\Phi(\lambda) \in \mathcal{INJ}(\ell_p)$ for all $\lambda \neq (0,0,0,\dots)$.
\vskip 3pt

Consequently, it only remains to prove that if \,$T = \Phi (\lambda )$ \,for some \,$\lambda \in \ell_p$ \,then the following is satisfied:
\begin{enumerate}
\item[$\bullet$] $T$ has not dense range.
\item[$\bullet$] $T$ is not cyclic.
\end{enumerate}

Since $1,2 \not\in \widetilde{\N}$, we obtain that for all $x \in \ell_p$ the first two coordinates of $T(x)$ are $0$.
Then $T(\ell_p) \subset \{(x_1,x_2,x_3, \dots ) \in \ell_p: \, x_1 = 0 = x_2\}$, and the latter set is a proper closed subset of \,$\ell_p$.
Thus, $T$ cannot have dense range.

\vskip 3pt

Finally, we argue similarly to the final part of the proof of Theorem \ref{Theorem-main-2}. Assume, by way of contradiction, that $T$ is cyclic. Then there would exist a vector $u = (u_1,u_2, \dots ) \in \ell_p$ enjoying the property that,
for every vector $v = (v_1,v_2, \dots ) \in \ell_p$, we can find a sequence \,$(w_j)_{j=1}^\infty \subset {\rm span}(\mathcal{O}(T;u))$ \,such that
\,$w_j \to v$ \,as \,$j \to \infty$. Observe that every $w_j$ has the form
$$
w_j = \mu_j u + \sum_{n=1}^{m(j)} c_{n,j} T^n u
$$
for certain scalars \,$\mu_j$ and $c_{n,j}$ ($n \in \{1, \dots, m(j)\}, \, j \in \N$).
Since \,$T^n u \in T(\ell_p)$ \,for all \,$n \in \N$, we obtain that the first two coordinates of \,$z_j := \sum_{n=1}^{m(j)} c_{n,j} T^n u$ \,are \,$0$.
If \,$\pi_1,\pi_2 : \ell_p \to \K$ \,represent, respectively, the projections on the first and the second coordinates, the continuity of \,$\pi_1,\pi_2$
\,entails for \,$s=1,2$ \,that
$$
\mu_j u_s = \mu_j u_s + 0 = \pi_s(\mu_j u) + \pi_s(z_j) = \pi_s(w_j)  \longrightarrow \pi_s(v) = v_s \hbox{ \ as \ } j \to \infty.
$$
If \,$u_1 = 0$, then taking \,$v = (1,0,0,0, \dots )$ \,yields \,$0 = \mu_j u_1 \to 1$, which is absurd.
If \,$u_1 \ne 0$, then taking \,$v = (0,1,0,0, \dots )$ \,yields \,$\mu_j u_1 \to 0$ (hence $\mu_j \to 0$) \,and \,$0 \leftarrow \mu_j u_2 \to 1$, which is again absurd.
This contradiction concludes the proof.
\end{proof}

\section{Spaceability of the class of non-$N$-supercyclic operators}

\quad The following dynamical concept was introduced and studied by Feldmann in \cite{feldman}.

\begin{definition}
{\em Let $X$ be a Fréchet space and $N \geq 1$. An operator $T \in \mathcal{L}(X)$ is said to be {\em $N$-supercyclic} if there exists an $N$-dimensional vector subspace $W$ such that the set
\[
\mathcal{O}_{W}(T) := \bigcup_{n=0}^\infty T^n(W)
\]
is dense in $X$.}
\end{definition}

Let $\mathcal{SC}^N(X) :=\{ T \in \mathcal{L}(X) : T \mbox{ is } N\mbox{-supercyclic}\}$ be the set of $N$-supercyclic operators over a Banach space $X$. It is obvious that $\mathcal{SC}^1(X) = \mathcal{SC}(X)$ and also that every $N$-supercyclic operator is $(N+1)$-supercyclic. In particular, supercyclic operators are always $N$-supercyclic for all $N\geq 1$. Moreover, the following relations
among classes of operators are clear:
\[
\begin{tikzcd}[column sep=3em, row sep=1.5em]
T \in \mathcal{HC}(X) \arrow[r, Rightarrow]
& T \in \mathcal{SC}(X) \arrow[d, Rightarrow] \arrow[r, Rightarrow]
& T \in \mathcal{CYC}(X) \\
& T \in \mathcal{SC}^N(X)  &
\end{tikzcd}
\]

All implications are strict. Examples of N-supercyclic operators that are neither supercyclic nor cyclic are provided in \cite[Section~3]{feldman} and \cite[Section~1]{bourdonshapiro}, respectively.
Furthermore, the authors provided examples of $(N+1)$-supercyclic operators that fail to be $N$-supercyclic.

\vskip 3pt

According to Corollary \ref{corollary}, the set $\mathcal{L}(X) \setminus \mathcal{SC}(X)$ is $SOT$-spaceable in $\mathcal{L}(X)$.
The following theorem provides a refinement of this result.

\begin{theorem}\label{theo5.2} Let $X$ be an infinite-dimensional Fréchet space and $N \geq1$. Then the set
\[
\mathcal{L}(X) \setminus \mathcal{SC}^N(X)
\]
is $SOT$-spaceable in $\mathcal{L}(X)$. In particular, it is norm-spaceable if \,$X$ is a Banach space.
\end{theorem}

\begin{proof} To avoid confusion, we will only detail the case $N=2$. Let $(x_n)_{n=1}^\infty$ be a basic sequence in $X$
(see Theorem \ref{MazurBessagaKadecPelczynski}) and $(e^*_n)_{n=1}^\infty$ be the coordinate functionals associated with the basis, extended to the whole space via the Hahn--Banach Theorem. Consider the operators $K_n(x)=e^*_n(x)x_n$ for $n \in \N$.
These operators are linearly independent because the $x_n$'s are.
Define the set 
\[
\mathcal{M} := \overline{{\rm span} \left(\{K_n : n \geq 4\} \right)}^{\rm SOT}.
\]
Trivially, $\mathcal{M}$ is an infinite-dimensional closed subspace of $\mathcal{L}(X)$.
It remains to show that $\mathcal{M}$
is composed of operators that are not $2$-supercyclic.

\vskip 3pt

With this aim, fix $T \in \mathcal{M}$. Then there exists a net $(T_\alpha)$ in ${\rm span}\{K_n : n \geq 4\}$ with $T_\alpha(x) \to T(x)$ for all $x \in X$. Note that $e^*_1 \circ T_\alpha = 0 = e^*_2 \circ T_\alpha =e^*_3 \circ T_\alpha$, from which it follows that
\[
e^*_1 \circ T^n= 0 = e^*_2 \circ T^n = e^*_3 \circ T^n \ \, (n \in \mathbb{N}).
\]

Suppose, by way of contradiction, that there exists a 2-dimensional vector space $W$ such that $\mathcal{O}_W (T)$ is dense in $X$.
Thus there exist sequences $(n_k)_{k}, (m_k)_k , (l_k)_k$ of nonnegative integers with
\begin{equation*}
T^{n_k}(v_{n_k}) \longrightarrow x_1 \,,
\quad 
T^{m_k}(w_{m_k}) \longrightarrow x_2 \,,
\quad \text{and} \quad
T^{l_k}(z_{l_k}) \longrightarrow x_3 \,,
\quad \text{as } \,k \to \infty,
\end{equation*}
where \,$v_{n_k},w_{m_k}, z_{m_k} \in W$. As usual, we set $T^0 :=$ the identity on $X$. Consider the set
$$
A := \{k \in \N: n_k \neq 0\}.
$$
Then $A$ is finite, otherwise we would have $T^{n_k}(v_{n_k}) \xrightarrow{k \in A} x_1$, and by the continuity of $e^*_1$ we would obtain
\[
0 = e^*_1 \circ T^{n_k}(v_{n_k}) \xrightarrow{k \in A} e^*_1(x_1) = 1,
\]
a contradiction.
Similarly, $\{k \in \N: m_k \neq 0\}$ and $\{k \in \N: l_k \neq 0\}$ are finite.
Then there exists \,$k_0 \in \N$ \,such that
$$
v_{n_k} = T^{n_k}(v_{n_k}) \quad w_{m_k} = T^{m_k}(w_{m_k}) \quad \hbox{and} \quad z_{l_k} = T^{l_k}(z_{l_k}) \,\, \hbox{ for all } \, k \ge k_0.
$$
Consequently, we obtain
\begin{equation*}
v_{n_k} \longrightarrow x_1 \quad  \quad w_{m_k} \longrightarrow x_2 \quad \hbox{and} \quad z_{l_k} \longrightarrow x_3 \quad \hbox{as } \, k \to \infty.
\end{equation*}
Since $W$ is finite dimensional, it is closed in $X$. Thus $x_i \in W$ for $i=1,2,3$, which is absurd because \,$\dim W =2$ \,and the system \,$\{x_1,x_2,x_3\}$ \,is linearly independent. This contradiction concludes the proof in the case $N = 2$.
The proof for a general $N \in \N$ is analogous, just by considering the set 
\[
\mathcal{M} := \overline{{\rm span} \left(\{K_n : n \geq N+2\} \right)}^{\rm SOT}.
\]
\vskip -18pt
\end{proof}

\vskip 5mm

\noindent\textbf{Acknowledgments and Funding.} We thank Professor Thiago Alves for fruitful conversations regarding on hypercyclicity and $N$-supercyclicity topics. Nacib G. Albuquerque was supported in part by CNPq Grants 406457/2023-9 and 403964/2024-5. Gustavo Ara\'ujo was partially supported by UEPB-PRPGP Nº 10/2025. Luis Bernal-Gonz\'alez thanks IMUS-Mar\'{\i}a de Maeztu grant CEX2024-001517-M - Apoyo a Unidades de Excelencia Mar\'{\i}a de Maeztu (funded by MICIU/AEI/10.13039/501100011033), and the Projects Programa operativo PPIT-FEDER Andaluc\'{\i}a 2021-2027 SOL2024-31596 and SOL\-2024-31708 for partially supporting this research. Evandio Dem\'etrio Jr. was supported by a FAPESQ-PB scholarship and CNPq Grant 406457/2023-9.

\vskip 5mm

\noindent\textbf{AI Disclosure Statement.} Generative AI tools were used solely for grammar checking and for correcting typographical and minor language errors. They were not used in the development of the mathematical content, results, or proofs of the manuscript.



\begin{thebibliography}{99}
\bibitem{aronbjms} {R.~Aron, L.~Bernal-González, P.~Jim\'enez-Rodríguez, G.~Mu\~{n}oz-Fernández and J.B.~Seoane-Sepúlveda}, \emph{On the size of special families of linear operators}, {Linear Algebra Appl.} {\bf 544} (2018), 186--205.

\bibitem{ABPS} R.M.~Aron, L. Bernal-Gonz\'alez, D.~Pellegrino and J.B.~Seoane-Sep\'ulveda, \emph{Lineability: The Search for Li\-ne\-a\-ri\-ty in Mathematics},
Monographs and Research Notes in Mathematics, CRC Press, Boca Raton, 2016.

\bibitem{Bayart} F.~Bayart and E.~Matheron, \emph{Dynamics of linear operators}, Cambridge University Press, 2009.

\bibitem{bernalopen} L.~Bernal-Gonz\'alez, \emph{The algebraic size of the family of injective operators}, {Open Math.}
{\bf 15} (2017), 13--20. 

\bibitem{beschan} J.~Bès and K.C.~Chan, \emph{Denseness of hypercyclic operators on a Fréchet space}, Houston J. Math. \textbf{29} (2003), 195--206.

\bibitem{beschan2} J.~Bès and K.C.~Chan, \emph{Approximation by chaotic operators and by conjugate classes}, J. Math. Anal. Appl. \textbf{284} (2003), 206--212.

\bibitem{bonetperis} J.~Bonet and A.~Peris, \emph{Hypercyclic operators on non-normable Fréchet spaces}, {J. Funct. Anal.} {\bf 159} (1998), 587--596.

\bibitem{botelho-aires} M.~Aires, G.~Botelho, \emph{Spaceability of sets of non-injective maps}, Bull. Braz. Math. Soc., New Series \textbf{56:34} (2025), 14 pages.

\bibitem{bagheri} A. R. Bagheri Salec, S. M. Tabatabaie and A.M Juweed, \emph{Spaceability of the set of surjective bounded operators of Banach function spaces}, J. Iran. Math. Soc. \textbf{6} (2025), no. 2, 107--113.


\bibitem{Chan} K.C.~Chan, \emph{The density of the hypercyclic operators on a Hilbert space}, J. Operator Theory {\bf 47} (1999), 231--244.


\bibitem{Diestel} J.~Diestel, \emph{Sequences and series in Banach spaces}, {Graduate Texts in Mathematics}, vol.~92, {Springer-Verlag}, {New York}, 1984.

\bibitem{diniz} {D.~Diniz, V.V.~F\'avaro, D.~Pellegrino and A.B.~Raposo},
\emph{Spaceability of the sets of surjective and injective operators between sequence spaces},
{Rev. Real Acad. Cien. Ex. F\'{\i}s. Nat. Ser. A Mat.}
{\bf 114:194} (2020), 11 pp.

\bibitem{feldman} N.S.~Feldman, \emph{$n$-supercyclic operators}, Studia Math. {\bf 151} (2) (2002), 141--159.

\bibitem{bourdonshapiro}{P.S.~Bourdon, N.S.~Feldman and J.H.~Shapiro}, \emph{Some properties of $N$-supercyclic operators}, Studia Math. \textbf{165} (2) (2004), 135--157.

\bibitem{fillmoresw} {P.A.~Fillmore, J.G.~Stampfli and J.P.~Williams}, \emph{On the numerical range, the essential spectrum, and a problem of Halmos},
Acta Sci. Math. (Szeged) \textbf{33} (1972), 179--192.


\bibitem{godefroyshapiro} {G.~Godefroy and J.H.~Shapiro}, \emph{Operators with dense, invariant, cyclic vectors manifolds},
{J. Funct. Anal.} \textbf{98} (1991), No.~2, 229--269.

\bibitem{grossesurvey}  K.-G.~Grosse-Erdmann, \emph{Universal families and
hypercyclic operators}, Bull. Amer. Math. Soc. {\bf 36} (1999), 345--381.

\bibitem{Grosse} K.-G.~Grosse-Erdmann and A.~Peris Manguillot, \emph{Linear chaos}, Springer Verlag, London, 2011.

\bibitem{Herrero} D.A.~Herrero, \emph{Limits of hypercyclic and supercyclic operators}, J. Funct. Anal. \textbf{99} (1991), 179--190.

\bibitem{Herzog} G. Herzog, \emph{On linear operators having supercyclic vectors}, Studia Math. \textbf{103} (1992), No.~3, 295--298. 

\bibitem{kamtangupta} P.K.~Kamtan and M.~Gupta, \emph{Theory of bases and cones}, Pitman, Boston, 1985.




\bibitem{prajitura} G.T.~Prajitura, \emph{The density of the hypercyclic operators in the strong operator topology},
Integr. Equat. Oper. Th. \textbf{49} (2004), 559--560.

\bibitem{rolewicz} S.~Rolewicz, \emph{On orbits of elements}, {Studia Math.} \textbf{32} (1969), 17--22.


\bibitem{RudinAF} W.~Rudin, \emph{Functional Analysis}, 2nd edition, McGraw-Hill Book Co., New York, 1991.

\bibitem{simon} B.~Simon, \emph{Operator Theory}, American Mathematical Society, Providence, Rhode Island, 2015.

\bibitem{Wu} P.Y.~Wu, \emph{Sums and products of cyclic operators}, Proc. Amer. Math. Soc. \textbf{122} (1994), 99--107.
\end{thebibliography}
\end{document}